\documentclass[12pt]{amsart}
\usepackage[latin1]{inputenc}
\usepackage{amsmath}
\usepackage{amsfonts}
\usepackage{amssymb}
\usepackage{amsthm}
\usepackage{relsize}
\usepackage{graphicx,enumitem,engrec,color}

\usepackage[alphabetic,initials]{amsrefs}
\usepackage{esint}

\usepackage{amssymb}
\usepackage{hyperref}
\usepackage{mathrsfs}
\usepackage{slashed}

\numberwithin{equation}{section}

\newtheorem{thrm}{Theorem}[section]
\newtheorem{lm}[thrm]{Lemma}
\newtheorem{cor}[thrm]{Corollary}
\newtheorem{defn}[thrm]{Definition}

\newtheorem{rmk}[thrm]{Remark}

\newcommand{\intB}{\int_{B_1^+}}

\newcommand{\xn}{{x_n}}
\begin{document}

\title[Regularity Results for a Penalized Boundary Obstacle Problem]{Regularity Results for a Penalized Boundary Obstacle Problem}

\author[D. Danielli]{Donatella Danielli}
\address[DD]{Department of Mathematics, Purdue University,  150 N. University St., West Lafayette, IN 47907}
\email{danielli@math.purdue.edu}

\author[R. Jain]{Rohit Jain}
\address[RJ]{Materiall, 500 E Calaveras, Suite 240, Milpitas, CA 95035}
\email{rohitjain19@gmail.com}

\dedicatory{
Al carissimo amico Sandro Salsa, con tanto affetto, ammirazione e gratitudine.
}

\begin{abstract}
In this paper we are concerned with a two-penalty boundary obstacle problem of interest in thermics, fluid dynamics and electricity. Specifically, we prove existence, uniqueness and optimal regularity of the solutions, and we establish structural properties of the free boundary.
\end{abstract}


\maketitle

\section{Introduction}

In this paper we study a penalized boundary obstacle problem of interest in thermics, fluid mechanics, and electricity. Given  a domain $\Omega$  in  $\mathbb{R}^n$, $n\geq 2$,  with sufficiently regular boundary $\partial\Omega=\Gamma_1\cup\Gamma_2$ and unit outer normal $\nu$, we consider the following stationary problem:

\begin{equation}
\begin{cases} \label{gen}
\Delta u &=f \qquad\text{  in } \Omega,\\
u &=g \qquad\text{ on } \Gamma_1\\
\frac{\partial u}{\partial \nu} &=-k_+\left((u-h)^+\right)^{p-1}+k_-\left((u-h)^-\right)^{p-1} \qquad\text{  on } \Gamma_2.
\end{cases}
\end{equation}

Here $f:\Omega\to \mathbb{R}$, $g:\Gamma_1\to\mathbb{R}$ and $h:\Gamma_2\to\mathbb{R}$ are given functions, $u^+=\max\{u,0\}$, $u^-=-\min\{u,0\}\geq 0$, $k_+$ and $k_-$ are non-negative constants, and $p>1$.  Our goal is to establish optimal regularity of the solutions, and to study properties of the  free boundary $\left(\partial \{u>h\}\cup\partial \{u<h\}\right)\cap\Gamma.$ We begin by observing that in the limiting case $k_+=k_-=0$,  $u$ is clearly the solution of a classical Neumann problem. The other limiting case, when $k^+=0$ and $k^-=+\infty$ (or equivalently $k^+=+\infty$ and $k^-=0$), is more interesting. The boundary condition, in fact, becomes
$$
u\geq h, \qquad \frac{\partial u}{\partial \nu}\geq 0,\qquad (u-h)\frac{\partial u}{\partial \nu}=0,
$$
 and $u$ is a solution of the \emph{Signorini problem}, also known as the \emph{thin obstacle problem}. The Signorini problem has received a resurgence of attention in the last decade, due to the discovery of several families of powerful monotonicity formulas, which in turn have allowed to establish the optimal regularity of the solution, a full classification of free boundary points, smoothness of the free boundary at regular points, and the structure of the free boundary at singular points. We refer  the interested reader to \cite{AC1}, \cite{ACS}, \cite{CSS}, \cite{GP}, \cite{DSS}, \cite{KPS}, see also the survey \cite{DS} and the references therein.

The general scheme of a solution to the Signorini problem provides a road map for the solution of problem \eqref{gen}, but there are two new substantial difficulties. The first one is due to the non-homogeneous nature of the boundary condition in \eqref{gen}, which in particular implies that this problem does not admit global homogeneous solutions of any degree. This is in stark contrast with the Signorini problem, where the existence and classification  of such solutions  play a pivotal role. Moreover, in the  thin obstacle problem it is readily seen that continuity arguments force $u$ to be always above $h$ (hence the nomenclature), whereas the case $h(x)>u(x)$ is no longer ruled out in \eqref{gen}. Allowing for both constants $k^+,\ k^-$ to be finite (even when one of the two vanishes) de facto destroys the one-phase character of the problem. In order to focus the attention on these new aspects, it is useful to understand first a simplified local version of \eqref{gen}, posed in the upper half ball
$$B_1^+ =\{x\in B_1 \mid x_n>0\},$$ with $f=h=0$. In this setting  problem \eqref{gen} becomes
\begin{equation}\label{statement_of_problem}
\left\lbrace\begin{aligned}
\Delta u &=0 \qquad\text{  in } B_1^+\\
u &=g \qquad\text{  on } (\partial B_1)^+\\
\frac{\partial u}{\partial x_n} &=k_+(u^+)^{p-1}-k_-(u^-)^{p-1} \qquad\text{  on } \Gamma.
\end{aligned}
\right.
\end{equation}
Here
\begin{align*}
(\partial B_1)^+ &= \{ x\in \partial B_1 \mid x_n >0\},\\
\Gamma &= \{ x\in B_1 \mid x_n=0\}.
\end{align*}
\noindent
An alternate perspective is given by the associated energy. We seek to minimize
\begin{equation}\label{energy}
J(v) =\frac{1}{2}\left(\int_{B_1} |\nabla v|^2\  dx+\int_\Gamma \left(\tilde{k}_-(v^-)^p+ \tilde{k}_+(v^+)^p\right) \ dx'\right)\
\end{equation}
over all $v \in W^{1,q}(B_1)$ with $q=\max\{2,p\}$ and $v-g\in W^{1,q}_0(B_1)$ for given boundary data $g$. Here $\tilde{k}_\pm = 2k_\pm/p$, and $x=(x', x_n)$. In this context we think of the data in~\eqref{statement_of_problem} as extended to all of $B_1$ by even reflection. A minimizer to this energy will be symmetric about $\Gamma$ and $u$ will correspond to the restriction to $B_1^+$.\\

Our first main result is the following:
\begin{thrm}\label{reg}
Let $g \in W^{1,q}(\Omega)$, $0\leq k_\pm <\infty,\ k_+\neq k_-$, and $p>1$. Then there exists a unique  minimizer $u\in W^{1,q}(B_1)$ of the energy $J(v)$ in \eqref{energy}. If $p$ is an integer, then $u\in C^{p-1,\alpha}(\overline{B_{1/2}^+})$ for every $\alpha<\min \{1,p-1\}$, and there exists a constant $C=C(n)>0$ such that
\begin{equation}\label{holder1}
\|u\|_{C^{p-1,\alpha}(B_{1/2}^+)}\leq C\left(\|u\|_{L^2(B_1^+)}+\|u\|_{L^p(\Gamma)}\right).
\end{equation}
If instead $p$ is not an integer, then  $u\in C^{\lfloor p-1\rfloor,\alpha}(\overline{B_{1/2}^+})$ for every $\alpha<p-1-\lfloor p-1\rfloor$, and there exists a constant $C=C(n)>0$ such that
\begin{equation}\label{holder2}
\|u\|_{C^{\lfloor p-1\rfloor,\alpha}(B_{1/2}^+)}\leq C\left(\|u\|_{L^2(B_1^+)}+\|u\|_{L^p(\Gamma)}\right).
\end{equation}
Additionally, if $p$ is a positive integer and $k_-=k_+$, or if $g$ does not change sign, then $u \in C^\infty(\overline{B_{1/2}^+})$.\\
\end{thrm}

In the case $p=2$, we can in fact establish that the regularity is optimal at points where the gradient does not vanish.
\begin{thrm}\label{optimal}
Let $u$ be the unique solution to \eqref{statement_of_problem} (see Definition \ref{def_weak}) when $p=2$. If $\nabla_{x'} u(x',0)\neq 0$, then $u$ is not in $C^{1,1}$ at $(x',0)$.
\end{thrm}

As an immediate consequence of the regularity of the solution and of the implicit function theorem, we obtain the following result on the regularity of the free boundary.
\begin{defn}
The {regular set of the free boundary} is defined as
$$
{\mathcal{R}(u)=\{(x',0)\in \Gamma\ |\ u(x',0)=0,\ \nabla_{x'} u(x',0)\neq 0\}}
$$
\end{defn}

\begin{thrm} Let $u$ be the unique solution to \eqref{statement_of_problem}, with $p>1$. If $x_0\in \mathcal{R}$, then in a neighborhood of $x_0$ the free boundary $\{ u(x',0)=0\}$ is a $C^{1,\alpha}-$ graph for all $\alpha<1$.
\end{thrm}

We next turn our attention to the study of the singular set. To this end, in what follows we assume $p\geq 2$.
\begin{defn} \label{sigma mu}
For
\begin{equation*}
N^{x_0}(r;u)=r\ \frac{\int_{B_r^+(x_0)} |\nabla u|^2\ dx}{\int_{(\partial B_r(x_0))^+}  u^2\ d\sigma(x) }
\end{equation*}
and
$$
\mu=N^{x_0}(0+;u)=\lim_{r\to 0}N^{x_0}(r;u),
$$
we define the set of singular points with frequency $\mu$ as
$$
\Sigma_\mu(u)=\{x_0\in\Gamma\ |\ u(x_0)=0, \ \nabla_{x'}u(x_0)=0, \text{ and } N^{x_0}(0+;u)=\mu\}.
$$
The dimension of $\Sigma_\mu(u)$ at a point $x\in \Sigma(u)$ is
$$
d_\mu^{x_0}=\dim\{\zeta\in\mathbb{R}^n\ |\ \langle\zeta,\nabla_{x'}p_\mu^{x_0}(x',0)\rangle=0 \mbox{ for all }x'\in\mathbb{R}^{n-1}\},
$$
where $p_\mu^{x_0}$ is a homogeneous polynomial of degree $\mu$ as in Theorem \ref{unique}.
Finally, we introduce
$$
\Sigma_\mu^d(u)=\{x_0\in\Sigma_\mu\ |\ d_\mu^{x_0}=d\}.
$$
\end{defn}
We observe here that the existence of the limit in Definition \ref{sigma mu} is guaranteed by Corollary \ref{cor0}, and that it follows from the proof of Theorem \ref{Almgren_blowup} below that $\mu$ is necessarily a positive integer. The structure of the singular set is described in the following result.
\begin{thrm}\label{structure}
Let $u$ be the unique solution to \eqref{statement_of_problem}, with $p\geq 2$. Then for every $\mu\in\mathbb{N}$ and $d=0,1, \dots,n-2$, the set $\Sigma_\mu^d(u)$ is contained in the countable union of $d$-dimensional $C^1$-submanifolds of $\Gamma$.
\end{thrm}

The proof of Theorem \ref{structure} follows the ideas of the corresponding result in \cite{GP} for the Signorini problem. It hinges on the monotonicity (or almost-monotonicity) of a perturbed Almgren functional and a Monneau-type functional (see Theorem \ref{almgren} and Corollary \ref{limMonn}). From these results we infer the growth rate and nondegeneracy of the solution near the free boundary. In turn, these properties allow to prove uniqueness and continuous dependance on the singular point of the blow-up limits. The rest of the proof is based on Whitney's extension and the implicit function theorem.

To conclude, we remark that considering a more general situation as in \eqref{gen} introduces significant technical difficulties. A standard approach, under suitable smoothness assumptions, consists in flattening $\Gamma_2$, which in turn leads to the study of a variable-coefficient operator and flat portion of the boundary. This problem, also with non-vanishing $h$, is the object of the recent paper \cite{DK}.

\subsection{Structure of the paper}
The paper is organized as follows. In Section 2 we describe some applications to problems of semi-permeable membranes and of temperature control, which motivate the study of \eqref{gen}. In Section 3 we establish existence and uniqueness of solutions, and prove Theorems \ref{reg} and \ref{optimal}. In Section 4 we prove the monotonicity of the perturbed functional of Almgren type, and infer some properties of the solution as a consequence. In Section 5 we introduce the Almgren rescalings, and discuss their blow-up limits. In Section 6 we prove the almost-monotonicity of a Monneau-type functional, and establish nondegeneracy of solutions. Finally, Section 7 is devoted to the proof of Theorem \ref{structure}.

\subsection{Acknowledgments} The authors wish to thank the anonymous referee, whose comments and suggestions helped to improve the readability of the paper.

 \section{Motivation}

 \subsection{Semi-permeable membranes} Following \cite[Section 2.2.2]{DL}, we briefly describe  the process of \emph{osmosis through semi-permeable walls}.  By $\Omega$ we denote a domain in  $\mathbb{R}^n$, $n\geq 2$,  with sufficiently regular boundary $\partial\Omega$. The region $\Omega$ consists of a porous medium occupied by a  viscous fluid which is only slightly compressible, and we denote its pressure field by $u(x)$. We assume that a portion  $\Gamma$ of $\partial\Omega$ consists of a semi-permeable membrane of finite  thickness, i.e. the fluid can freely enter in $\Omega$, but the outflow of fluid is prevented. Combining the law of conservation of mass with Darcy's law, one finds that $u$ satisfies the equation
$$
\Delta u- \frac{\partial u}{\partial t} =f\mbox{ in }\Omega,
$$
where $f=f(x,t)$ is a given function. When a fluid pressure $h(x)$, for $x\in\Gamma$, is applied to $\Gamma$ on the outside of $\Omega$, one of two cases holds:
$$
h(x)<u(x,t)\mbox{ or } h(x)\geq u(x,t).
$$
In the former, the semi-permeable wall prevents the fluid from leaving $\Omega$, so that the flux  is null. If we let $\nu$ denote the outer unit normal to $\Gamma$, we then have
\begin{equation}\label{case1}
\frac{\partial u}{\partial \nu}=0.
\end{equation}
In the latter case, the fluid enters $\Omega$. It is reasonable to assume the outflow to be proportional to the difference in pressure, so that
\begin{equation}\label{case2}
-\frac{\partial u}{\partial \nu} =k(u-h),
\end{equation}
where $k>0$ measures the conductivity of the wall. Combining  \eqref{case1} and \eqref{case2}, we obtain the boundary condition
\begin{equation}\label{bdry}
\frac{\partial u}{\partial \nu} =k(u-h)^-\mbox{ on }\Gamma.
\end{equation}

In our model \eqref{gen}, we allow  for fluid flow to occur both into and out of $\Omega$ with different permeability constants, under the assumption that  the flux in each direction is proportional to a power of the pressure.

\subsection{Temperature control} An alternative interpretation of the model is as a \emph{boundary temperature control problem}, which we only briefly outline here.   We assume that a continuous medium occupies a region $\Omega$ in $\mathbb{R}^n$, with boundary $\Gamma$ and outer unit normal $\nu$. Given a reference temperature $h(x)$, for $x\in \Gamma$, it is required that the temperature at the boundary $u(x,t)$ deviates as little as possible from $h(x)$. To this end, thermostatic controls are placed on the boundary to inject an appropriate heat flux when necessary. The controls are regulated as follows:
\begin{itemize}
\item[(i)] If $u(x,t)=h(x)$, no correction is needed and therefore the heat flux is null.
\item[(ii)] If $u(x,t)\neq h(x)$, a quantity of heat proportional to the difference between $u(x,t)$ and $h(x)$ is injected.
\end{itemize}
We can thus write the boundary condition as
$$
-\frac{\partial u}{\partial \nu}=\Phi(u),
$$
where
\begin{equation*}
\Phi(u)=\begin{cases}
k_-(u-h)\qquad&\text{ if } u< h\\
0&\text{ if }u=h\\
k_+(u-h)\qquad&\text{ if } u> h
\end{cases}
\end{equation*}

More in general, one can assume that $\Phi(u)$ is a continuous and increasing function of $u$. For further details, we refer to  \cite[Section 2.3.1]{DL}, see also  \cite{AC1}  for the limiting case $k_-=0$ and $k_+=+\infty$ and \cite{ALP} for the case $p=1$ in \eqref{energy}.

\section{Optimal regularity of solutions}

We begin this section by proving existence and uniqueness of minimizers to~\eqref{energy}.
We let $\mathcal{K} =\{v\in W^{1,2}(B_1) \mid v-g \in W^{1,2}_0(B_1) \}$.
\begin{lm}\label{exist_minimizer}
There exists a unique minimizer $u \in \mathcal{K}$ for the energy $J(v)$ given by~\eqref{energy}.
\end{lm}
\begin{proof} Throughout this proof we will pass to subsequences whenever necessary without comment. Let $u_l$ be a minimizing sequence. Then $\Vert \nabla u_l \Vert_2$ is clearly bounded owing to the form of the energy itself. By using the Poincar\'e inequality on $u_l-g$ we deduce that the sequence $u_l$ is bounded in the $W^{1,2}(B_1)$ norm. Thus there exists a weak limit $u$ which is necessarily in $\mathcal{K}$.
We may assume that $u_l \to u$ in $L^2$ and a.e. The weak convergence of $u_l$ to $u$ in $W^{1,2}$ and the strong  convergence in $L^2$ imply that
\[
\int_{B_1} |\nabla u|^2\ dx\leq \liminf_{l\to\infty}\int_{B_1} |\nabla u_l|^2\ dx.
\]
This clearly follows from the property of weak convergence
 \[
 \Vert u \Vert_{W^{1,2}(B_1)} \leq \liminf_{l\to\infty} \Vert u_l \Vert_{W^{1,2}(B_1)}
 \]
 and, because of the strong $L^2$ convergence, the inequality must fall on the gradient part of the norm.

 To prove that $u$ is a minimizer we must show then that
 \[
 \int_\Gamma (u^\pm)^p\ dx' \leq \liminf_{l\to\infty} \int_\Gamma (u_l^\pm)^p\ dx'.
 \]
It will suffice to demonstrate this for $u^-$; the result for $u^+$ is proved in an analogous fashion. The trace operator $T : W^{1,2}(B_1^+)\to L^2(\partial B_1^+)$ is a bounded linear operator, since the half ball is a Lipschitz domain. Furthermore, in this setting it is a compact operator, and thus takes weakly convergent sequences to strongly convergent ones. Suppressing the $Tu_l$ notation and simply writing $u_l$ we then have that
\[
u_l \to u \quad \text{in } L^2(\Gamma).
\]
From this we may assume that $u_l \to u$ a.e. on $\Gamma$. But then clearly $(u^-_l)^p \to (u^-)^p$ a.e. and applying Fatou's Lemma we have
\[
\int_\Gamma (u^-)^p\ dx' \leq \liminf \int_\Gamma (u_l^-)^p\ dx'
\]
which completes the proof of existence.

Uniqueness follows by observing that $(f+g)^\pm \leq f^\pm+g^\pm$, and then applying standard arguments.
 \end{proof}

 Next, we recall the definition of a weak solution (see [L]):
 \begin{defn}\label{def_weak}
We say that $u$ is a weak solution to
\begin{equation*}
\begin{cases}
\Delta u &=0  \text{ in } B_1^+\\
u_\xn &=f \text{ on } \Gamma
\end{cases}
\end{equation*}
if for every $\xi\in C^\infty(B_1^+)$ vanishing on $(\partial B_1)^+$ we have
\[
\intB \nabla u \nabla \xi\ dx = -\int_\Gamma f\xi \ dx'
\]
\end{defn}

It is easy to show that the minimizer $u$ is a weak solution to our problem.

\begin{lm}\label{lemma.weak}
The minimizer $u$ obtained in Lemma \ref{exist_minimizer} is a weak solution to \eqref{statement_of_problem}. That is,
\begin{equation}\label{weak_solution}
\int_{B_1^+} \nabla u \nabla \xi\ dx = -\int_\Gamma ( -k_-(u^-)^{p-1}+k_+(u^+)^{p-1})\xi\ dx'
\end{equation}
for all $\xi \in C^\infty(B_1^+)$ vanishing on $(\partial B_1)^+$.
\end{lm}
\begin{proof}
This is a standard variational fact. See for example the proof of Lemma 4.1 in [ALP].
\end{proof}
\begin{rmk} The $-k_-$ term appears since $u^- = -\min\{u,0\}$.
\end{rmk}

We now turn to the regularity of the solution. Our strategy will be to first prove an initial H\"older regularity which will improve afterwards. The first step  is an energy estimate for $u$.

\begin{lm}\label{energy_est}
Let $u$ be the minimizer of \eqref{energy}. Then we have for any $B_{2r} \subset B_1$
\[
\int_{B_r} |\nabla u|^2\ dx \leq \frac{c}{r^2}\int_{B_2r} u^2 \ dx.
\]
\end{lm}
 \begin{proof}
We first prove the corresponding estimate for $u^-=-\min\{u,0\}$. Let $\eta\in C_0^\infty(B_{2r})$ with
 \[
 \eta \equiv 1 \text{ in } B_r, \quad |\nabla \eta |\leq \frac{c}{r^2}.
 \]
Taking $\xi = u^-\eta^2$ and using~\eqref{weak_solution} we have
\begin{align*}
\int_{B_1} \nabla u \nabla (u^-\eta^2)\ dx &= -\int_\Gamma ( -k_-(u^-)^{p-1}+k_+(u^+)^{p-1})u^-\eta^2\ dx'\\
&=-\int_\Gamma ( -k_-(u^-)^{p-1})u^-\eta^2\ dx' \geq 0.
\end{align*}
Expanding yields
 \[
 \int_{B_1}\left(\eta^2\nabla u\nabla u^-  +2u^-\eta \nabla u\nabla\eta\right)\ dx=  \int_{B_1}-|\nabla u^-|^2 \eta^2 -2u^-\eta \nabla u^-\nabla\eta \ dx \geq 0
 \]
 or
 \[
  \int_{B_1}|\nabla u^-|^2 \eta^2 \ dx\leq -\int_{B_1}2u^-\eta \nabla u^-\nabla\eta\ dx.
 \]
At this point standard energy arguments imply
 \[
 \int_{B_r} |\nabla u^-|^2 \ dx\leq \frac{c}{r^2} \int_{B_{2r}} (u^-)^2 \ dx.
 \]
 A similar argument implies the same inequality with $u^+$; together they yield the energy estimate for $u$.
 \end{proof}

 Next, we use the energy estimate to prove an initial H\"older modulus of continuity for $u$. This regularity is much lower than optimal, but it will allows us to bootstrap to obtain higher regularity.

\begin{lm}\label{Holder_cont}
The  solution to \eqref{statement_of_problem} is in $C^{0,1/2}(\overline{B_{1/2}})$.
\end{lm}
\begin{proof}
 Let $B_r := B_r(x)$ for $r<1/4$ and $x\in B_1$, and let $v$ be the harmonic replacement of $u$ in $B_r$. Set $\Gamma_r = B_r\cap \Gamma$.
By minimality we have
\begin{equation}\label{Holder_cont:eq1}
\int_{B_r} \left(|\nabla u|^2 - |\nabla v|^2\right)\ dx \leq \int_{\Gamma_r} \left(k_-((v^-)^p-(u^-)^p)+k_+((v^+)^p-(u^+)^p)\right)\ dx'.
\end{equation}
However, since $v$ is harmonic we have
\[
\int_{B_r} \nabla v\cdot\nabla (v-u)\ dx= 0,
\]
and thus
\begin{equation}\label{Holder_cont:eq2}
\int_{B_r} |\nabla u - \nabla v|^2\ dx =\int_{B_r} \left(|\nabla u|^2 - |\nabla v|^2\right) \ dx.
\end{equation}

Next, since $v$ is the harmonic lifting of $u$, $|v| \leq |u|$ in $B_r$. In turn, the computation used in Lemma~\ref{energy_est} demonstrated that $u^\pm$ are subharmonic, and therefore $|u| = u^++u^-$ is as well. Thus, by the maximum principle, $\sup_{B_1} |u| \leq \sup_{\partial B_1} |u| = \sup g$, the given boundary data in~\eqref{statement_of_problem}. In particular,
\[
\int_{\Gamma_r} \left(k_-((v^-)^p-(u^-)^p)+k_+((v^+)^p-(u^+)^p)\right)\ dx' \leq Cr^{n-1},
\]
with $C$ independent of $x$ and $v$.
From this fact, combined with \eqref{Holder_cont:eq1} and~\eqref{Holder_cont:eq2}, we infer
\[
\int_{B_r} |\nabla u - \nabla v|^2\ dx \leq Cr^{n-1}.
\]
At this point, we can mimic the derivation in \cite[Theorem 3.1]{AP} to deduce that
\[
\int_{B_r} |\nabla u|^2\ dx \leq Cr^{n-1}.
\]
In turn Morrey's Dirichlet Growth Theorem (see for instance \cite[Corollary 9.1.6]{J}) implies the desired H\"older-1/2 regularity inside $B_{1/2}$.
\end{proof}

We have reached the proof of our main result:
\begin{proof}[Proof of Theorem \ref{reg}]

Existence and uniqueness follow from Lemma \ref{exist_minimizer}. Thus, we need only to show the desired regularity. From Lemma~\ref{lemma.weak} we know that $u$ is a weak solution to our problem on $B_1^+$. Moreover,
 $u$ is $C^{0,1/2}(\overline{B_{1/2}})$ by Lemma~\ref{Holder_cont}. The $`\pm'$ operation preserves H\"older regularity (with the same H\"older norm) so $u^\pm\in C^{0,1/2}(\overline{B_{1/2}})$ and in particular on the thin region $\Gamma$. This implies that
 \begin{equation}\label{MainThrm:eq1}
-k_-(u^-)^{p-1}+k_+(u^+)^{p-1}
 \end{equation}
is H\"older continuous of order $\gamma$, although $\gamma$ will in general not be $1/2$.

 Nevertheless, this implies that $u$ is a weak solution to an oblique derivative problem with H\"older continuous boundary data, namely $-k_-(u^-)^{p-1}+k_+(u^+)^{p-1}$. Regularity theory for such a problem (see e.g. \cite[Proposition 5.53]{L}) then yields that $u$ must be $C^{1,\gamma}$ up to the boundary, with
 $$
 |u|_{1+\gamma}\leq C\left(\sup|u|+|u|_\gamma\right).
 $$
 But in turn this implies that $u$ is Lipschitz up to the boundary, in which case~\eqref{MainThrm:eq1} is H\"older continuous of order $p-1$ when $p\leq 2$; if $p>2$ this is to be interpreted as differentiablity with a H\"older modulus of continuity. Applying the regularity theory once again we have the result of the theorem.

Now suppose that $g$ does not change sign. We aim to show that $u$ does not change sign either, in which case $u^\pm =u$ (and thus $u^\pm$ is as smooth as $u$ is) and the regularity result above can be bootstrapped to prove that $u$ is smooth. To this end, suppose that $g\geq 0$, but $u$ attains a minimum value which is negative, say $u(z) =m <0$. Then $z$ must lie on $\Gamma$. In particular, $z\in\Gamma_R=\Gamma\cap B_R$ for some $0<R<1$. Now, trivial modifications to the above arguments allow to show $u\in C^{1,\alpha}(\overline{B_R})$, and therefore we can assume that the restriction of  $u$ to $\Gamma_R$  is $C^{1,\alpha}$. Next, we apply the Hopf Lemma. Since $u$ is harmonic in the interior we must have
\[
\frac{\partial u}{\partial \nu}(z) <0.
\]
Here $\nu$ is the outer normal vector, which at the point $z$ is $-e_n$. Thus
\[
\frac{\partial u}{\partial {x_n}}(z) >0.
\]
However, the boundary condition along $\Gamma$ is given by
\[
\frac{\partial u}{\partial x_n} = k_+(u^+)^{p-1}-k_-(u^-)^{p-1},
\]
which holds in a classical sense since $u$ is $C^{1,\alpha}$ in a neighborhood of $z$. But $u(z)<0$, and therefore the boundary condition at $z$ is $u_{x_n} =-k_-u^-(z)<0$, a contradiction. We have thus shown that,  if $g\geq 0$, $u$ cannot be negative along $\Gamma$. As a consequence, $u$ is non-negative everywhere, so that $u^\pm = u$ and higher regularity follows by bootstrapping.

A similar argument shows that if $g\leq 0$ then $u\leq 0$ everywhere, which again implies higher regularity. Finally, the case $p$ integer and $k_+=k_-$ follows immediately from a repeated application of \cite[Proposition 5.53]{L}.

\end{proof}

We now show that, at least in the case $p=2$, the regularity obtained in Theorem \ref{reg} is optimal at points where the gradient of $u$ is non-vanishing.

\begin{proof}[Proof of Theorem \ref{optimal}] We argue by contradiction, and assume that $u\in C^{1,1}(0)$, with $\nabla u(0)\neq 0$. Thanks to Theorem \ref{reg}, we know that $u$ has a unique differential $P=\nabla u(0)$. Without loss of generality, we may assume that $P$ is also a superdifferential for $u^-$ (if not, consider $u^+$). We refer, for instance, to \cite[Chapter 3]{CaSi} for the definition and properties of superdifferentials. We begin by observing that we can write
$$
\frac{\partial u}{\partial x_n} = k_+u^+-k_+u^-+k_+u^--k_-u^-=k_+u+(k_+-k_-)u^-.
$$
Thus,
$$
(k_+-k_-)u^-=\frac{\partial u}{\partial x_n}-k_+u.
$$
From this, applying the extension theorem in \cite{CS} (with a slight abuse of notation, $u(x')$ denotes the restriction of $u(x)=u(x',x_n)$ to $x_n=0$) and the semigroup property of $(-\Delta)^s$, we deduce
\begin{align}
\notag(k_+-k_-)[-(-\Delta_{x'})^{1/2} u^-(x')]&=[-(-\Delta_{x'})^{1/2}]\circ [-(-\Delta_{x'})^{1/2}]u(x') -k_+[-(-\Delta_{x'})^{1/2}]u(x')\\
&=\Delta_{x'} u(x')-k_+\frac{\partial u}{\partial x_n}(x',0).\label{eq.opt}
\end{align}
  Because of our $C^{1,1}$ assumption, we have  that $C_0\leq u_{\tau\tau}(0)\leq C_1$ for some constants $C_0, C_1>0$ and for any tangential direction $\tau$. Hence, keeping also Theorem \ref{reg} in mind, it follows from \eqref{eq.opt}
 $$
| -(-\Delta_{x'})^{1/2} u^-(0)|\leq C_2
 $$
 for some $C_2>0$.  We now consider
$$
\psi(x)=\left[u^-(0)+\min\{P\cdot x, 0\}+\frac{C_1}{2}|x|^2\right]\chi_{B_1}.
$$
A straightforward computation yields
$$
-(-\Delta_{x'})^{1/2}\psi(0)=-\infty.
$$
In addition, $u^-(x)\leq \psi(x)$, with equality at $x=0$. From the definition of $(-\Delta)^{1/2}$ , we infer
$$
-(-\Delta_{x'})^{1/2}u^-(0)\leq -(-\Delta_{x'})^{1/2}\psi(0)=-\infty.
$$
But we showed above that $-(-\Delta_{x'})^{1/2}u^-(0)\geq -C_2$. We have thus reached a contradiction.
\end{proof}

\section{Monotonicity of a perturbed Almgren frequency functional}

In this section we establish some properties of the solution around  free boundary  points in the case $p\geq 2$. For $u$ solution to \eqref{statement_of_problem}, we define the \emph{coincidence set} $\Lambda(u)=\{(x',0)\ |\ u(x',0)=0\}$, and the \emph{free boundary} $\mathcal{F}(u)=\partial\Lambda(u)$. In the Signorini problem, the monotonicity of the Almgren's Frequency Functional
\begin{equation}\label{N_reg}
N(r;u)=N(r)=r\ \frac{\int_{B_r^+} |\nabla u|^2\ dx}{\int_{(\partial B_r)^+}  u^2\ d\sigma(x) }
\end{equation}
plays a fundamental role in the study of both the solution and the free boundary. In our setting, $N(r)$ may fail to be monotone, but a suitable perturbation is. We thus introduce the  \emph{perturbed Almgren Frequency Functional} at the point $x_0=0$ as
\begin{equation}\label{N_tilde}
\tilde{N}(r;u) =\tilde{N}(r)=r\ \frac{\int_{B_r^+} |\nabla u|^2\ dx+\frac{2}{p}\int_{\Gamma_r} F(u)\ dx'}{\int_{{(\partial B_r)}^+} u^2 \ d\sigma(x)},
\end{equation}
with $F(u) =k_-(u^-)^p +k_+(u^+)^p$ and $B_r=B_r(0)$.
\begin{thrm}\label{almgren} Let $u$ be a solution to \eqref{statement_of_problem}, with $p\geq 2$. Then $\tilde{N}(r;u)$ is monotone increasing in $r\in(0,1)$.
\end{thrm}

\begin{proof} Let
\[
H(r) =\int_{{(\partial B_r)}^+} u^2\ d\sigma(x), \quad D(r) = \int_{B_r^+} |\nabla u|^2 \ dx.
\]
We begin by observing
\begin{equation}
H'(r) =\frac{n-1}{r}H(r) +2\int_{{(\partial B_r)}^+} uu_\nu \ d\sigma(x).
\end{equation}
We also have
\begin{align*}
D(r) &:= \int_{B_r^+} |\nabla u|^2\ dx = \int_{B_r^+} (|\nabla u|^2 +u\Delta u)\ dx\\
&= \int_{B_r^+} \Delta(\frac{u^2}{2})\ dx = \int_{{(\partial B_r)}^+} uu_\nu \ d\sigma(x)+\int_{\Gamma_r}uu_\nu\ dx'\\
&= \int_{{(\partial B_r)}^+} uu_\nu\ d\sigma(x) +\int_{\Gamma_r} [k_-(u^-)^{p-1} -k_+(u^+)^{p-1}]u\ dx'\\
&= \int_{{(\partial B_r)}^+} uu_\nu\ d\sigma(x)  - \int_{\Gamma_r} [ k_+(u^+)^p +k_-(u^-)^p]\ dx'.
\end{align*}
By Rellich's Identity
\begin{align*}
D'(r) &=\int_{{(\partial B_r)}^+} |\nabla u|^2\ d\sigma(x)\\
&=\frac{n-2}{r}\int_{B_r^+} |\nabla u|^2\ dx +2\int_{{(\partial B_r)}^+} u_\nu^2\ d\sigma(x) -\frac{2}{r}\int_{\Gamma_r} \langle x,\nabla u\rangle u_{x_n}\ dx'\\
& = \frac{n-2}{r}\int_{B_r^+} |\nabla u|^2 \ dx+2\int_{{(\partial B_r)}^+} u_\nu^2 \ d\sigma(x)-\frac{2}{r}\int_{\Gamma_r} \langle x,\nabla u\rangle(-k_-(u^-)^{p-1}+k_+(u^+)^{p-1})\ dx' ,
\end{align*}
which we can rewrite as
\begin{align}\label{eq.D1}
D'(r) &= \frac{n-2}{r}D(r) +2\int_{{(\partial B_r)}^+} u_\nu^2 \ d\sigma(x)-\frac{2}{pr}\int_{\Gamma_r} [k_-\langle x, \nabla(u^-)^p \rangle +k_+\langle x,\nabla(u^+)^p\rangle]\ dx'.
\end{align}
Using integration by parts we note that
\[
\int_{\Gamma_r} \langle x,\nabla(u^\pm)^p \rangle\ dx' = \int_{{\partial\Gamma_r}} r(u^\pm)^p\ d\sigma(x') -(n-1)\int_{\Gamma_r} (u^\pm)^p\ dx'.
\]
Applying this fact in \eqref{eq.D1} we obtain
\begin{align}\label{eq.D2}
D'(r) &= \frac{n-2}{r}D(r) +2\int_{{(\partial B_r)}^+} u_\nu^2\ d\sigma(x) \\
&-\frac{2}{pr}\left[\int_{\partial \Gamma_r} r(k_-(u^-)^p +k_+(u^+)^p)\ d\sigma(x') -(n-1) \int_{\Gamma_r}(k_-(u^-)^p+k_+(u^+)^p) \right]\ dx'.\notag
\end{align}
For the sake of brevity we will define
\[
\tilde{D}(r) =D(r) +\frac{2}{p}\int_{\Gamma_r}F(u)\ dx'.
\]
A direct computation, together with \eqref{eq.D1} and \eqref{eq.D2}, yields
\begin{align*}
\frac{\tilde{N}'(r)}{\tilde{N}(r)} &=\frac{1}{r}+ \frac{\tilde{D}'(r)}{\tilde{D}(r)}-\frac{H'(r)}{H(r)}\\
&= \frac{1}{r}+\frac{D'(r)+\frac{2}{p}\int_{\partial \Gamma_r} F(u)\ d\sigma(x')}{D(r)+\frac{2}{p} \int_{\Gamma_r}F(u)\ dx'}-\frac{n-1}{r}-2\frac{\int_{{(\partial B_r)}^+}uu_\nu \ d\sigma(x)}{\int_{{(\partial B_r)}^+} u^2\ d\sigma(x)}\\
&= \frac{1}{r}+\frac{n-2}{r}\frac{D(r)}{D(r) +\frac{2}{p}\int_{\Gamma_r} F(u)\ dx'}+2\frac{\int_{{(\partial B_r)}^+} u_\nu^2\ d\sigma(x)}{D(r) +\frac{2}{p}\int_{\Gamma_r}F(u)\ dx'}\\
&+\frac{2(n-1)}{pr}\frac{\int_{\Gamma_r} F(u)\ dx'}{D(r)+\frac{2}{p}\int_{\Gamma_r}F(u)\ dx'}
+\frac{1-n}{r} -2\frac{\int_{{(\partial B_r)}^+} uu_\nu\  d\sigma(x)}{\int_{{(\partial B_r)}^+} u^2 \ d\sigma(x)}.
\end{align*}
Collecting terms we have
\begin{align}\label{eq.N1}
\frac{\tilde{N}'(r)}{\tilde{N}(r)}=&\frac{1}{r}\left[1-\frac{D(r)}{D(r) +\frac{2}{p}\int_{\Gamma_r} F(u)\ dx'} \right]\\
&+2\left[\frac{\int_{{(\partial B_r)}^+} u_\nu^2\ d\sigma(x)}{D(r) +\frac{2}{p}\int_{\Gamma_r} F(u)\ dx'}-\frac{\int_{{(\partial B_r)}^+}uu_\nu\ d\sigma(x)}{\int_{{(\partial B_r)}^+} u^2\ d\sigma(x)}\right] \notag\\
&+\frac{n-1}{r}\left[ \frac{D(r)}{D(r) +\frac{2}{p}\int_{\partial \Gamma_r}F(u)\ dx'}+\frac{2}{p}\frac{\int_{\Gamma_r} F(u)\ dx'}{D(r) +\frac{2}{p}\int_{\Gamma_r}F(u)\ dx'}-1\right] \notag
\end{align}
Clearly the first term in \eqref{eq.N1} is non-negative, whereas the last one vanishes. On the other hand, from \eqref{statement_of_problem} we know
\begin{equation}\label{weak}
\int_{{(\partial B_r)}^+} uu_\nu\ d\sigma(x)\ = D(r) +\int_{\Gamma_r} F(u)\ dx' \geq D(r) +\frac{2}{p}\int_{\Gamma_r} F(u)\ dx',
\end{equation}
since $p\geq 2$. In turn this implies
\begin{equation*}
\frac{\int_{{(\partial B_r)}^+} u_\nu^2\ d\sigma(x)}{D(r) +\frac{2}{p}\int_{\Gamma_r} F(u)\ dx'}-\frac{\int_{{(\partial B_r)}^+} uu_\nu\ d\sigma(x)}{\int_{{(\partial B_r)}^+} u^2\ d\sigma(x)}
\geq \frac{\int_{{(\partial B_r)}^+} u_\nu^2 \ d\sigma(x)}{\int_{{(\partial B_r)}^+} uu_\nu\ d\sigma(x)}-\frac{\int_{{(\partial B_r)}^+} uu_\nu\ d\sigma(x)}{\int_{{(\partial B_r)}^+} u^2\ d\sigma(x)}\geq 0
\end{equation*}
 by the Cauchy-Schwartz inequality. Hence $\frac{\tilde{N}'(r)}{N(r)}\geq 0$, and the proof is complete.
\end{proof}

We now state some consequences of Theorem \ref{almgren}. The first result shows that, even if the Almgren's Frequency Functional $N(r)$ in \eqref{N_reg} fails to be monotone, it still has a limit as $r\to 0^+$, and in fact its limit coincides with the one of $\tilde{N}(r)$. In order to prove this result, we will need the following trace-type inequality (see, for instance, \cite[Lemma 2.5]{FF}).

\begin{lm}\label{Poincare} Let $u\in W^{1,2}(B_r^+)$. Then there is a bounded linear function $T: W^{1,2}(B_r^+)\to L^2(\partial B_r^+)$ such that $T(u)$ is the restriction of $u$ to $\partial B_r^+$ for any $u\in C^1(\overline{B_r^+})$. Moreover, there exists a  constant $C>0$ such that
\begin{equation}\label{trace}
\int_{\Gamma_r}u^2\ dx'\leq C\left(r\int_{B_r^+}|\nabla u|^2+\int_{(\partial B_r)^+} u^2\ d\sigma(x)\right).
\end{equation}
\end{lm}

\begin{cor}\label{cor0} Let $N(r)$ and $\tilde{N}(r)$ be given by \eqref{N_reg} and \eqref{N_tilde}, respectively. Define $\mu~=~\lim_{r\to 0^+} \tilde{N}(r)$. Then there exists $N(0+):=\lim_{r\to 0^+} {N}(r)$, and $N(0+)=\mu$.
\end{cor}
\begin{proof} We begin by observing that, since $\tilde{N}(r)\geq 0$, Theorem~\ref{almgren} guarantees that $\mu$ exists, and that $\mu\in [0,\infty)$. Since $F(u)\geq 0$, trivially
\begin{equation}\label{above}
N(r)\leq \tilde{N}(r).
\end{equation}
On the other hand, if we let $k=\max\{k^+, k^-\}$ and $0<r<1/2$,
$$
\int_{\Gamma_r}F(u)\ dx'\leq k \int_{\Gamma_r} |u|^p\ dx' \leq k \sup_{B_{1/2}^+}|u|^{p-2}\int_{\Gamma_r} |u|^2\ dx'.
$$
Applying \eqref{trace} we get
\begin{equation}\label{p_trace}
\int_{\Gamma_r}F(u)\ dx'\leq C\left(r\int_{B_r^+}|\nabla u|^2+\int_{(\partial B_r)^+} u^2\ d\sigma(x)\right).
\end{equation}
Using the notations introduced in the proof of Theorem~\ref{almgren}, we thus obtain
$$
\tilde{N}(r)\leq N(r)+Cr^2\frac{D(r)}{H(r)}+Cr=(1+Cr)N(r)+Cr.
$$
Hence,
\begin{equation}\label{below}
N(r)\geq\frac{\tilde{N}(r)-Cr}{1+Cr},
\end{equation}
and the desired conclusion follows from \eqref{above} and \eqref{below}.
\end{proof}

Next, we introduce the quantity
$$\varphi(r) = \varphi(r ; u) = \fint_{(\partial B_{r})^{+}} u^{2}.$$
\begin{cor}\label {cor1}
Let $\mu=\lim_{r\to 0^+} \tilde{N}(r)\in[0,\infty)$. The following hold:
\begin{itemize}
\item[(a)] The function $r\mapsto r^{-2\mu}\varphi(r)$ is nondecreasing for $0 < r < 1/2$. In particular,
$$ \varphi(r) \leq (r/2)^{2 \mu} \varphi(1/2) \leq C_n (r/2)^{2\mu} \sup_{B_{1/2}^+} |u|^2,$$
where $C_n>0$ is a dimensional constant.
\item[(b)] Let $0 < r < 1/2$. Then for any $\delta > 0$ there exists $R_0=R_{0}(\delta) > 0$ such that for all $r<R \leq R_{0}$
$$\varphi(R) \leq e^{2C\left(1-\frac{2}{p}\right)(\mu+\delta+1)(R-r)}\left(\frac{R}{r}\right)^{2(\mu + \delta)} \varphi(r).$$
Here $C$ is the constant appearing in \eqref{p_trace}.
\end{itemize}
\end{cor}

\begin{proof}
We begin the proof of (a) by  computing
$$\varphi'(r) = \frac{d}{dr} \fint_{(\partial B_{r})^+} u^{2} =  2 \fint_{(\partial B_{r})^+} u u_{\nu}.$$
Hence, taking \eqref{weak} into account, we have
\[
\begin{split}
\frac{d}{dr} (r^{-2 \mu} \varphi(r)) & = -2 \mu r^{-2\mu-1} \varphi(r) +\frac{2r^{-2\mu}}{|(\partial B_r)^+|}\left( \int_{B_{r}^+} |\nabla u|^{2} +  \int_{\Gamma_{r}} F(u)\ dx'\right) \\
&  = \frac{2 r^{-2\mu-1}}{| (\partial B_{r})^+ |} \left(-\mu \int_{(\partial B_{r})^{+}} u^{2}+r \int_{B_{r}^+} |\nabla u|^{2}+ r \int_{\Gamma_{r}} F(u)\ dx'\right)  \\
& = \frac{2 r^{-2\mu-1}}{| (\partial B_{r})^+ |} \left(-\mu \int_{(\partial B_{r})^{+}} u^{2}+r \int_{B_{r}^+} |\nabla u|^{2}+ \frac{2}{p}r \int_{\Gamma_{r}} F(u)\ dx'\right)\\
&\qquad+\frac{2 r^{-2\mu}}{| (\partial B_{r})^+ |}\left(1-\frac{2}{p}\right)\int_{\Gamma_{r}} F(u)\ dx'\geq 0.
\end{split}
\]
In the last inequality we have used Theorem \ref{almgren} and the fact that $p\geq 2$.

For the proof of (b), we compute
\begin{align*}
\frac{r}{2} \frac{\varphi '(r)}{\varphi(r)}&= r\frac{\int_{{(\partial B_r)}^+} uu_\nu\ d\sigma(x)}{\int_{{(\partial B_r)}^+} u^2\ d\sigma(x)}\\
(\text{by \eqref{weak})}\qquad\qquad &=r\frac{D(r) +\int_{\Gamma_r} F(u)\ dx'}{\int_{{(\partial B_r)}^+} u^2\ d\sigma(x)}\\
&=\tilde{N}(r)+r\left(1-\frac{2}{p}\right)\frac{\int_{\Gamma_r} F(u)\ dx'}{\int_{{(\partial B_r)}^+} u^2\ d\sigma(x)}\\
(\text{by \eqref{p_trace})}\qquad\qquad &\leq \tilde{N}(r)+Cr\left(1-\frac{2}{p}\right)\left(N(r)+1\right).
\end{align*}
Thanks to Corollary \ref{cor0}, there exists $R_0=R_{0}(\delta)>0$ such that $N(r)\leq \tilde{N}(r) \leq \mu + \delta$ for $r< R \leq R_{0}.$ We then have
$$ \frac{d}{dr} \log \varphi(u) \leq\frac{2}{r} (\mu + \delta)+2C\left(1-\frac{2}{p}\right)(\mu+\delta+1).$$
To conclude we integrate the inequality over $(r, R).$
\end{proof}

\begin{cor}\label{cor2} Let $u$ be a solution to \eqref{statement_of_problem}. Then, for all $x \in B_{r}^+,\ 0<r<1/2$,
$$|u(x)| \leq C_n(r/2)^{\mu} \sup_{B_{1/2}^+} |u|,$$
where $C_n>0$ is a dimensional constant.
\end{cor}

\begin{proof}
We note that $(u^+)^2$ is a positive subharmonic function in the domain. Hence,
$$(u^{+})^{2}(0) \leq \fint_{(\partial B_{r})^+} (u^{+})^{2} \leq C_n (r/2)^{2\mu} \sup_{B_{1/2}^+} |u|^{2},$$
by Corollary \ref{cor1}(a). A similar estimate holds for $(u^{-})^{2}$.
\end{proof}

\section{Almgren rescalings and blow-ups}

The next step in our analysis is to study blow-up sequences around a free boundary point $x_0\in\mathcal{F}(u)$. Without loss of generality, we may assume $x_0=0$. We define, for $0<r<1$, the \emph{Almgren rescalings}
\begin{equation}\label{blowup}
v_r(x)=\frac{u(rx)}{[\varphi(r;u)]^{1/2}}.
\end{equation}
We note that $\|v_{r}\|_{L^{2}((\partial B_{1})^+)} = 1$. Moreover, for $R_0=R_0(\delta)$ as in Corollary \ref{cor1}(b), a fixed $R>1$, and every $r>0$ such that $rR\leq R_0$, we have, thanks to Corollaries~\ref{cor0} and~\ref{cor1}(b),
$$
\int_{B^+_{R}} |\nabla v_{r}|^{2}\ dx = R^{n-2}N(rR; u) \frac{\varphi(rR;u)}{\varphi(r,u)}\leq C(\mu+\delta)R^{n-2+2(\mu+\delta)}.
$$
Hence, after an even reflection across $\{x_n=0\}$, any sequence $\{v_{r_j}\}$, with $r_j\to 0^+$ as $j \to\infty$,  is equibounded in $H^{1}_{loc}(\mathbb{R}^n)$, and by Theorem \ref{reg}, it is also bounded in $C^{1,\alpha}_{loc}(\mathbb{R}^n).$ Thus, there exists a subsequence, denoted by $v_j$, and  a function $v^*$ (which we will refer to as the \emph{Almgren blow-up}), such that
$$v_{j} \to v^{*}\qquad\text{and} \qquad\nabla v_{j} \to \nabla v^{*}\qquad\text{as }j\to\infty,
$$
uniformly on every compact subset of $\mathbb{R}^{n}$.
We note that the fact $\|v_{j}\|_{L^{2}((\partial B_{1})^+)} = 1$ in particular  implies that the blow-up is nontrivial. In addition, by rescaling,
$$
\mu=\lim_{j\to \infty} N(r_j;u)=\lim_{j\to \infty}N(1;v_j)=\lim_{j\to \infty}\int_{B_1^+}|\nabla v_j|^2\ dx=\int_{B_1^+}|\nabla v^*|^2\ dx,
$$
and therefore we have (keeping in mind that $u(0)=0$) $\mu>0$. A similar rescaling argument, in fact, shows that for any $\rho>0$
\begin{equation}\label{hom}
N(\rho;v^*)=\lim_{j\to \infty}N(\rho;v_j)=\lim_{j\to \infty}N(r_j\rho;u)=\mu
\end{equation}
Next, for a function $\xi\in C^\infty_0(B_1)$ we compute
\begin{align}\label{rescaling1}
\int_{B_1^+}\nabla v_r(x)\nabla \xi(x)\ dx&=\frac{r}{[\varphi(r;u)]^{1/2}}\int_{B_1^+}\nabla u(rx)\nabla \xi(x)\ dx\\
\text{(by \eqref{weak_solution})}\qquad\qquad &=\frac{r}{[\varphi(r;u)]^{1/2}}\int_{\Gamma} ( -k_-(u^-)^{p-1}(rx',0)+k_+(u^+)^{p-1}(rx',0))\xi(x,0)\ dx'.\notag
\end{align}
We now assume $p\geq 3$. An application of \eqref{trace} yields
\begin{align}\label{rescaling2}
\left|\int_{B_1^+}\nabla v_r(x)\nabla \xi(x)\ dx\right|&\leq C \frac{r}{[\varphi(r;u)]^{1/2}}\sup_{B_r^+}|u|^{p-3} \int_{\Gamma} u^2(rx',0)\ dx' \\
 &\leq C \frac{r^{2-n}}{[\varphi(r;u)]^{1/2}}\sup_{B_r^+}|u|^{p-3}  \int_{\Gamma_r} u^2(x',0)\ dx' \notag\\
&\leq C \frac{r^{2-n}}{[\varphi(r;u)]^{1/2}}\sup_{B_r^+}|u|^{p-3} \left(r\int_{B_r^+}|\nabla u(x)|^2\ dx+\int_{(\partial B_r)^+}u^2(x)\ d\sigma(x)\right) \notag\\
&\leq C r [\varphi(r;u)]^{1/2}\sup_{B_r^+}|u|^{p-3}\left(N(r;u)+1\right) \notag.
\end{align}
Since $[\varphi(r;u)]^{1/2}\leq Cr^\mu$ by Corollary~\ref{cor1}(a) and $\sup_{B_r^+}|u|\leq Cr^\mu$ by Corollary~\ref{cor2}, we conclude that the last term in \eqref{rescaling2} goes to zero as $r\to 0^+$. Hence, if we extend $v^*$ by even reflection across $\{x_n=0\}$, then $v^*$ is harmonic in $B_1$. The same conclusion can be reached in the case $2\leq p<3$ by applying H\"older's inequality in the last integral in \eqref{rescaling1}. Now, it is well known (see, for instance \cite[Section 9.3.1]{PSU}) that a function harmonic in $B_1$ and satisfying \eqref{hom} is necessarily an homogeneous harmonic polynomial of degree $\mu\in\mathbb{N}$ (since we have already ruled out the possibility $\mu=0$). If we also assume $\nabla_{x'}u(0)=0$, from the uniform convergence of $\nabla v_j$ to $\nabla v^*$ we deduce $\mu\geq 2$. We have thus proved the following.
\begin{thrm}\label{Almgren_blowup} Let $u$ be a solution to \eqref{statement_of_problem}, with $u(0)=0$ and $\nabla_{x'}u(0)=0$. If $v_r$ is as in \eqref{blowup}, then for any sequence $r_j\to 0^+$ there exists a subsequence $\{v_j\}$ of $\{v_{r_j}\}$ and a function $v^*$ such that
$$
v_{j} \to v^{*}\quad\text{in }H^1(B_1^+)\text{ and in } C^1(B_1^+).
$$
Furthermore, the even reflection of  $v^*$ across $\{x_n=0\}$ is an homogeneous harmonic polynomial of degree $\mu=N(0+;u)\in\mathbb{N},\ \mu\geq 2$.
\end{thrm}

\section{A Monneau-type monotonicity formula}

Our next step consists in establishing  almost-monotonicity of a functional of Monneau type. Using the notations introduced in the proof of Theorem \ref{almgren}, we define the \emph{Weiss functional}
$$
W_\mu(r;u)=\frac{H(r;u)}{r^{n-1+2\mu}}(N(r;u)-\mu).
$$
\begin{thrm}\label{Monneau} Let $u$ be as in Theorem \ref{Almgren_blowup}, and let $p_\mu$ be an harmonic polynomial, homogeneous of degree $\mu$ and even in $x_n$. If we define the \emph{Monneau functional} as
$$
M_\mu(r;u,p_\mu)=\frac{1}{r^{n-1+2\mu}}\int_{(\partial B_r)^+}(u-p_\mu)^2 \ d\sigma(x),
$$
then there exists $C>0$ such that
\begin{equation}\label{monn1}
\frac{d}{dr}\left(M_\mu(r;u,p_\mu)+Cr\right)\geq \frac{2}{r}W_\mu(r;u).
\end{equation}
\end{thrm}
\begin{proof}
Let $p_\mu$ be an harmonic polynomial, homogeneous of degree $\mu$ and even in $x_n$. Since $N(r;p_\mu)=\mu$, we have $W_\mu(r;p_\mu)=0$.  We now rewrite
$$
W_\mu(r;u)=\frac{1}{r^{n-2+2\mu}}D(r;u)-\frac{\mu}{r^{n-1+2\mu}}H(r;u),
$$
and let $w=u-p_\mu$. Then
\begin{align*}
W_\mu(r;u)&=W_\mu(r;u)-W_\mu(r;p_\mu)\\
&= \frac{1}{r^{n-2+2\mu}}\int_{B_r^+}\left(|\nabla w|^2+2\nabla w\cdot \nabla p_\mu\right)\ dx-\frac{\mu}{r^{n-1+2\mu}}\int_{(\partial B_r)^+}(w^2+2p_\mu w)\ d\sigma(x).
\end{align*}
Integrating by parts in the first integral, keeping in mind that $p_\mu$ is harmonic and $\frac{\partial p_\mu}{\partial x_n}=0$ on $\Gamma_r$, we obtain
\begin{align*}
W_\mu(r;u)&= \frac{1}{r^{n-2+2\mu}}\int_{(\partial B_r)^+}w\nabla w\cdot\frac{x}{r}\ d\sigma(x)-\frac{1}{r^{n-2+2\mu}}\int_{\Gamma_r}\frac{\partial u}{\partial x_n}(u-p_\mu)\ dx'\\
&+\frac{2}{r^{n-2+2\mu}}\int_{(\partial B_r)^+}w\nabla p_\mu\cdot \frac{x}{r}\ d\sigma(x)-\frac{\mu}{r^{n-1+2\mu}}\int_{(\partial B_r)^+}(w^2+2p_\mu w)\ d\sigma(x).
\end{align*}
Noting that $\nabla p_\mu\cdot x=\mu p_\mu$, we infer
\begin{align}\label{final}
W_\mu(r;u)&=\frac{1}{r^{n-1+2\mu}}\int_{(\partial B_r)^+}w\nabla w\cdot x\ d\sigma(x)-\frac{1}{r^{n-2+2\mu}}\int_{\Gamma_r}\frac{\partial u}{\partial x_n}(u-p_\mu)\ dx'\\
&-\frac{\mu}{r^{n-1+2\mu}}\int_{(\partial B_r)^+}w^2\ d\sigma(x).\notag
\end{align}
We now observe the following facts:
\begin{equation}\label{I}
 \frac{1}{r^{n-1+2\mu}}\int_{(\partial B_r)^+}w(\nabla w\cdot x-\mu w)\ d\sigma(x)=\frac{r}{2}\frac{d}{dr}\left(\frac{1}{r^{n-1+2\mu}}\int_{(\partial B_r)^+}w^2\ d\sigma(x)\right),
 \end{equation}
 \begin{equation}\label{II}
 \int_{\Gamma_r}u\frac{\partial u}{\partial x_n}\ dx'=\int_{\Gamma_r}F(u)\ dx'\geq 0,
 \end{equation}
  \begin{equation}\label{III}
 \int_{\Gamma_r}p_\mu\frac{\partial u}{\partial x_n}\ dx'\leq Cr^{n-1+p\mu}.
 \end{equation}
 In \eqref{III} we have used the boundary condition in \eqref{statement_of_problem}, Corollary \ref{cor2}, and the fact that $p_\mu$ is homogeneous of degree $\mu$. As a consequence, the constant in \eqref{III} will depend on $\sup_{B_{1/2}^+}|u|$ and $\|p_\mu\|_{L^1(\Gamma)}$. Using \eqref{I}-\eqref{III} in \eqref{final}, we obtain
 \begin{equation}\label{c1}
 W_\mu(r;u)\leq \frac{r}{2}\frac{d}{dr}\left(\frac{1}{r^{n-1+2\mu}}\int_{(\partial B_r)^+}w^2\ d\sigma(x)\right)+Cr^{1+\mu(p-2)}.
 \end{equation}
An application of \eqref{c1} yields
 \begin{equation*}
 \frac{d}{dr}\left(\frac{1}{r^{n-1+2\mu}}\int_{(\partial B_r)^+}w^2\right)\geq \frac{2}{r}W_\mu(r,u)-Cr^{\mu(p-2)}\geq \frac{2}{r}W_\mu(r;u)-C,
 \end{equation*}
 thus concluding the proof.
 \end{proof}

 \begin{cor}\label{limMonn} Under the assumption of Theorem \ref{Monneau}, there exists $C>0$ such that
$$
\frac{d}{dr}\left(M_\mu(r;u,p_\mu)+Cr\right)\geq 0.
$$
In particular, there exists $\lim_{r\to 0^+} M_\mu(r;u,p_\mu).$
\end{cor}
\begin{proof}
Thanks to the inequality \eqref{below} and Theorem \ref{almgren}, we have for $0<r<1/2$
$$
N(r;u)\geq \frac{\tilde{N}(r;u)-Cr}{1+Cr}\geq \frac{\mu-Cr}{1+Cr}
$$
and therefore
\begin{equation}\label{eq1}
N(r;u)-\mu\geq -\frac{C(\mu+1)r}{1+Cr}.
\end{equation}
Thus, using \eqref{eq1} and Corollary \ref{cor1}(a), we obtain
\begin{equation}\label{W}
W_\mu(r;u)\geq -\frac{C(\mu+1)r}{1+Cr}\frac{H(r;u)}{r^{n-1+2\mu}}\geq -Cr.
\end{equation}
Inserting this information in \eqref{monn1} gives the desired conclusion.
 \end{proof}

With Corollary \ref{limMonn} at our disposal, we can prove nondegeneracy of the solution at free boundary points.

\begin{lm}\label{l_nondeg} Let $u$ be as in Theorem \ref{Almgren_blowup}. There exists $C>0$ and $0<R_0<1$, depending possibly on $u$, such that
\begin{equation}\label{nondeg}
\sup_{(\partial B_r)^+}|u|\geq Cr^\mu
\end{equation}
for all $0<r<R_0$.
\end{lm}

\begin{proof} Arguing by contradiction, assume that \eqref{nondeg} does not hold. We thus have, for a sequence $r~=~r_j\to~0^+$,
\begin{equation}\label{contr}
\varphi(r)=o(r^{2\mu}).
\end{equation}
Possibly passing to a subsequence, Theorem \ref{Almgren_blowup} guarantees that the Almgren rescalings $u_{r}(x)$ introduced in \eqref{blowup} converge uniformly, as $r\to 0^+$, to a nontrivial harmonic polynomial $p_\mu$, homogeneous of degree $\mu$ and even in $x_n$. We now compute $M_\mu(0+;u,p_\mu)=\lim_{r\to 0}M_\mu(r;u,p_\mu)$, whose existence follows from Corollary \ref{limMonn}. We have
\begin{equation}\label{lim1}
M_\mu(r;u,p_\mu)=\frac{1}{r^{n-1+2\mu}}\int_{(\partial B_r)^+}u^2\ d\sigma(x)+\frac{1}{r^{n-1+2\mu}}\int_{(\partial B_r)^+}(-2up_\mu+p_\mu^2)\ d\sigma(x).
\end{equation}
We observe that the first integral in \eqref{lim1} goes to $0$ as $r\to 0^+$ because of \eqref{contr}. Moreover, the homogeneity of $p_\mu$ implies
\begin{equation}\label{lim2}
\frac{1}{r^{n-1+2\mu}}\int_{(\partial B_r)^+}p_\mu^2\ d\sigma(x)=\int_{(\partial B_1)^+}p_\mu^2\ d\sigma(y),
\end{equation}
and therefore
\begin{equation}\label{lim3}
\frac{1}{r^{n-1+2\mu}}\int_{(\partial B_r)^+}|up_\mu|\ d\sigma(x)\leq
\left(\frac{1}{r^{n-1+2\mu}}\int_{(\partial B_r)^+}u^2\ d\sigma(x)\right)^{1/2}\left(\frac{1}{r^{n-1+2\mu}}\int_{(\partial B_r)^+}p_\mu^2\ d\sigma(x)\right)^{1/2}\to 0.
\end{equation}
Combining \eqref{lim1}-\eqref{lim3} we infer
$$
M_\mu(0+;u,p_\mu)=\int_{(\partial B_1)^+}p_\mu^2\ d\sigma(y)=\frac{1}{r^{n-1+2\mu}}\int_{(\partial B_r)^+}p_\mu^2\ d\sigma(x)
$$
for all $0<r<1/2$.
An application of Corollary \ref{limMonn} then yields
$$
\frac{1}{r^{n-1+2\mu}}\int_{(\partial B_r)^+}(u-p_\mu)^2 \ d\sigma(x)+Cr\geq M_\mu(0+;u,p_\mu)=\frac{1}{r^{n-1+2\mu}}\int_{(\partial B_r)^+}p_\mu^2\ d\sigma(x),
$$
which we can rewrite as
$$
\frac{1}{r^{n-1+2\mu}}\int_{(\partial B_r)^+}(u^2-2up_\mu) \ d\sigma(x)\geq -Cr.
$$
Rescaling according to \eqref{blowup}, we obtain
$$
\frac{1}{r^{2\mu}}\int_{(\partial B_1)^+}\left(\varphi(r)v_r^2-2[\varphi(r)]^{1/2}r^\mu v_r p_\mu\right)\ d\sigma(x)\geq -Cr,
$$
or equivalently
\begin{equation}\label{limfin}
\int_{(\partial B_1)^+}\left(\frac{[\varphi(r)]^{1/2}}{r^\mu}v_r^2-2v_r p_\mu\right)\ d\sigma(x)\geq -C\frac{r^{\mu+1}}{[\varphi(r)]^{1/2}}.
\end{equation}
At this point we observe that, thanks to Corollary \ref{cor1}(b), for each $0<\delta<1$ there exist $C=C(p,\ \mu,\ \delta)$, and $R_0=R_0(\delta)>0$ such that $\varphi(r)\geq C_1r^{2(\mu+\delta)}$ for all $0<r<R_0$. Hence, letting $r\to 0^+$ in \eqref{limfin} we conclude
$$
-\int_{(\partial B_1)^+}p_\mu^2\geq 0,
$$
which is a contradiction since $p_\mu$ is nonzero.
\end{proof}

\begin{cor}\label{F-sigma} Let $u$ and $\mu$ be as in Theorem \ref{Almgren_blowup}. The set $\Sigma_\mu(u)$ (see Definition \ref{sigma mu}) is of type $F_\sigma$, i.e. it is the union of countably many closed sets.
\end{cor}
\begin{proof} The proof follows the lines of the one of Lemma 1.5.3 in \cite{GP}, and it is omitted.
\end{proof}

\section{The structure of the singular set}

To continue with our analysis, we introduce the \emph{homogeneous rescalings}
\begin{equation}\label{homresc}
v_r^{(\mu)}(x)=\frac{u(rx)}{r^\mu},\qquad 0<r<1,
\end{equation}
and show existence and uniqueness of the blow-ups with respect to this family of rescalings.
\begin{thrm} \label{unique} Let $u$ and $\mu$ be as in Theorem \ref{Almgren_blowup}. If $v_r^{(\mu)}$ is as in \eqref{homresc}, then there exists a unique function $v_0$ such that
$$
v_r^{(\mu)}(x) \to v_0\quad\text{in }H^1(B_1^+)\text{ and in } C^1(B_1^+).
$$
Furthermore, the even reflection of  $v_0$ across $\{x_n=0\}$ is an homogeneous harmonic polynomial of degree $\mu$.
\end{thrm}
\begin{proof} By Corollary \ref{cor1}(a) and Lemma \ref{l_nondeg}, there exist constants $C_1,\ C_2>0$ such that
$$
C_1 r^\mu\leq [\varphi(r)]^{1/2}\leq C_2 r^\mu, \qquad 0<r<1.
$$
From this it follows that any limit of the rescalings $v_r^{(\mu)}$ over any sequence $r_j\to 0^+$ is a positive multiple of the Almgrens rescalings $v_r$ as in \eqref{blowup}. By Theorem \ref{Almgren_blowup}, we know that the even reflection of  $v_0$ across $\{x_n=0\}$ is an homogeneous harmonic polynomial of degree $\mu$, and that the convergence is both in  $H^1(B_1^+)\text{ and in } C^1(B_1^+)$. At this point, we only need to show uniqueness. To this end, we apply Corollary \ref{limMonn} with $p_\mu=v_0$. We thus have
$$
M_{\mu}(0+,u, v_0)=\lim_{r_j\to 0^+}M_{\mu}(r_j,u,v_0)=\lim_{r_j\to 0^+}\int_{(\partial B_1)^+}(u_{r_j}^{(\mu)}-u_0)^2 \ d\sigma(x)=0,
$$
where the last equality follows from the first part of the proof. In particular, we have that
$$
M_{\mu}(r,u,v_0)=\int_{(\partial B_1)^+}(u_r^{(\mu)}-u_0)^2 \ d\sigma(x)\to 0
$$
as $r\to 0^+$, and not only over $r_j\to 0^+$. If $v_0'$ is a limit of $v_r^{(\mu)}$ over another sequence $r_j'\to 0^+$, we infer that
$$
\int_{(\partial B_1)^+}(u_0'-u_0)^2 \ d\sigma(x)=0.
$$
Hence, $u_0'=u_0$ and the proof is complete.
\end{proof}
The next step consists in showing the continuous dependance of the blow-ups. In what follows, we denote by $\mathcal{P}_\mu$, with $\mu\in\mathbb{N}$, the class of harmonic polynomials homogeneous of degree $\mu$ and even in $x_n$.

\begin{thrm}\label{contdep} Let $u$ be a solution to \eqref{statement_of_problem}, $\mu\in\mathbb{N}$ with $\mu\geq 2$, and $x_0\in \Sigma_\mu(u)$. Denote by $p_\mu^{x_0}$ the blow-up of $u$ at $x_0$ as in Theorem \ref{unique}, so that
$$
u(x)=p_\mu^{x_0}(x-x_0)+o(|x-x_0|^\mu).
$$
Then the mapping $x_0\mapsto p_\mu^{x_0}$ from $\Sigma_\mu(u)$ to $ \mathcal{P}_\mu$ is continuous. Moreover, for any compact set $K\subset \Sigma_\mu(u)\cap B_1$ there exists a modulus of continuity $\omega_\mu$, with $\omega_\mu(0+)=0$, such that
$$
\left|u(x)-p_{\mu}^{x_0}(x-x_0)\right|\leq \omega_\mu(|x-x_0|)|x-x_0|^\mu
$$
for any $x_0\in K.$
\end{thrm}
\begin{proof}
We begin by observing that $\mathcal{P}_\mu$ is a convex subset of the finite-dimensional vector space of all polynomials homogeneous of degree $\mu$, and therefore all norms are equivalent. We choose to endow it with the norm in $L^2((\partial B_1)^+)$. We begin by fixing $x_0\in\Sigma_\mu(u)$ and $\varepsilon >0$ sufficiently small. Then there is $r_\varepsilon=r_\varepsilon(x_0)>0$ such that
$$
M_\mu^{x_0}(r_\varepsilon;u,p_\mu^{x_0}):=\frac{1}{r_\varepsilon^{n-1+2\mu}}\int_{(\partial B_{r_\varepsilon})^+}\left(u(x+x_0)-p_\mu^{x_0}\right)^2 d\sigma(x)<\varepsilon.
$$
In turn, there exists $\delta_\varepsilon=\delta_\varepsilon(x_0)>0$ such that if $x_1\in\Sigma_\mu(u)\cap B_{\delta_\varepsilon}(x_0)$, then
\begin{equation}\label{intermediate}
M_\mu^{x_1}(r_\varepsilon;u,p_\mu^{x_0})=\frac{1}{r_\varepsilon^{n-1+2\mu}}\int_{(\partial B_{r_\varepsilon})^+}\left(u(x+x_1)-p_\mu^{x_0}\right)^2 d\sigma(x)<2\varepsilon.
\end{equation}
Corollary \ref{limMonn} yields that $M_\mu^{x_1}(r;u,p_\mu^{x_0})<3\varepsilon$, provided that $0<r<r_\varepsilon$ and $r_\varepsilon$ is small enough. Rescaling and passing to the limit as $r\to 0^+$ , we obtain
\begin{equation}\label{int2}
\int_{(\partial B_1)^+}\left(p_\mu^{x_1}-p_\mu^{x_0}\right)^2 d\sigma(x)=M_\mu^{x_1}(0+;u,p_\mu^{x_0})\leq 3\varepsilon
\end{equation}
and the first part of the theorem is proved. In order to establish the second part, we observe that, for $|x_1-x_0|<\delta_\varepsilon$ and $0<r<r_\varepsilon$, combining \eqref{intermediate} and \eqref{int2} we obtain
\begin{align*}
\|u(\cdot+x_1)-p_\mu^{x_1}\|_{L^2((\partial B_r)^+)}&\leq \|u(\cdot+x_1)-p_\mu^{x_0}\|_{L^2((\partial B_r)^+)}+\|p_\mu^{x_0}-p_\mu^{x_1}\|_{L^2((\partial B_r)^+)}\\
&\leq 2(3\varepsilon)^{1/2}r^{\frac{n-1}{2}+\mu}.\notag
\end{align*}
Integrating in $r$ we obtain
\begin{equation}\label{int3}
\|u(\cdot+x_1)-p_\mu^{x_1}\|_{L^2(( B_{r})^+)}\leq C\varepsilon^{1/2}r^{n/2+\mu},
\end{equation}
with $C=C(n,\mu)>0$. Letting
$$
v_{r,x_1}^{(\mu)}(x)=\frac{u(rx+x_1)}{r^\mu},
$$
we infer from \eqref{int3}
\begin{equation}\label{int4}
\|v_{r,x_1}^{(\mu)}(x)-p_\mu^{x_1}\|_{L^2( B_1^+)}\leq C\varepsilon^{1/2}.
\end{equation}
At this point we observe that the difference $w_r=v_{r,x_1}^{(\mu)}(x)-p_\mu^{x_1}$ is a weak solution to
\begin{equation*}
\begin{cases}
\Delta w_r=0&\qquad \text{ in } B_1^+,\\
\mathlarger{\frac{\partial w_r}{\partial \nu}}=r\left[-k_+\left(\left(v_{r,x_1}^{(\mu)}\right)^+\right)^{p-1}+k_-\left(\left(v_{r,x_1}^{(\mu)}\right)^-\right)^{p-1}\right] &\qquad\text{  on } \Gamma.
\end{cases}
\end{equation*}
By the $L^\infty-L^2$ interior estimates (see, for instance \cite[Theorem 5.36]{L}), there exists a positive constant $C=C(n, k_+, k_-)$ such that, for some $q>n-1$,
\begin{equation}\label{final1}
\|v_{r,x_1}^{(\mu)}(x)-p_\mu^{x_1}\|_{L^\infty(( B_{1/2})^+)}\leq C\left(\|v_{r,x_1}^{(\mu)}(x)-p_\mu^{x_1}\left\|_{L^2(( B_{1})^+)}+r\||v_{r,x_1}^{(\mu)}|^{p-1}\right\|_{L^{q}(\Gamma)}\right).
\end{equation}
To estimate the right-hand side in \eqref{final1}, we recall that $|v_{r,x_1}^{(\mu)}|\leq C$ by Corollary \ref{cor2}, and thus
\begin{equation}\label{final2}
r\||v_{r,x_1}^{(\mu)}|^{p-1}\|_{L^{q}(\Gamma)}\leq C r^{1+\frac{n-1}{q}}.
\end{equation}
Combining \eqref{final1} with \eqref{int4} and \eqref{final2}, we obtain
\begin{equation}\label{final3}
\|v_{r,x_1}^{(\mu)}(x)-p_\mu^{x_1}\|_{L^\infty(( B_{1/2})^+)}\leq C_\varepsilon
\end{equation}
for $0<r<r_\varepsilon$ sufficiently small, and $C_\varepsilon\to 0$ as $\varepsilon\to 0$. To conclude, we cover the compact set $K\subset \Sigma_\mu(u)\cap B_1$ with a finite number of balls $B_{\delta_\varepsilon(x_0^i)}(x_0^i)$ for some choice of $x_0^i\in K$, $i=1,\dots, N$. Hence, for $r<r_\varepsilon^K:=\min\{r_\varepsilon(x_0^i)|i=1,\dots, N\}$, we have that \eqref{final3} holds for all $x_1\in K$. The desired conclusion readily follows.
\end{proof}
\begin{proof}[Proof of Theorem \ref{structure}] The proof of the structure of $\Sigma_\mu^d$ is centered on Corollary \ref{F-sigma}, Theorem \ref{contdep}, Whitney's extension theorem, and the implicit function theorem. Since the arguments are essentially identical to the ones in the proof of Theorem 1.3.8 in \cite{GP}, we omit the details and refer the interested reader to that source.
\end{proof}

\end{document}